\documentclass[12pt]{amsart}
\usepackage{graphicx}
\usepackage{amssymb}
\usepackage{epstopdf}
\newtheorem{theorem}{Theorem}

\newtheorem{proposition}[theorem]{Proposition}

\title{On the Realization of Double Occurrence Words}
\author{B.\ Shtylla}
\author{L.\ Traldi}
\author{L.\ Zulli}
\address{Dept.\ of Mathematics, Univ.\ of Utah, Salt Lake City, UT 84112}
\email{shtyllab@math.utah.edu}
\address{Dept.\ of Mathematics, Lafayette College, Easton, PA 18042}
\email{traldil@lafayette.edu}
\address{Dept.\ of Mathematics, Lafayette College, Easton, PA 18042}
\email{zullil@lafayette.edu}

\date{09/17/07}
\begin{document}

\begin{abstract} 
Let $S$ be a double occurrence word, and let $M_S$ be the word's interlacement matrix, regarded as a matrix over $GF(2)$. Gauss addressed the question of which double occurrence words are realizable by generic closed curves in the plane. We reformulate answers given by Rosenstiehl and by de~Fraysseix and Ossona~de~Mendez to give new graph-theoretic and algebraic characterizations of realizable words.
Our algebraic characterization is especially pleasing: $S$ is realizable if and only if there exists a diagonal matrix $D_S$ such that $M_S+D_S$ is idempotent over $GF(2)$.
\end{abstract}

\maketitle
\setcounter{section}{0}

\section{Introduction}

A \textit{double occurrence word} is a finite string of symbols in which each symbol appears precisely twice. For example,  $adbacdcb$ is a double occurrence word in the symbols $a,b,c,d$. Two distinct symbols in a double occurrence word are said to be \textit{interlaced} if each appears precisely once between the two occurrences of the other. For example, in the word above, $a$ and $b$ are interlaced but $b$ and $c$ are not. For each double occurrence word $S$ we have a simple graph $\Lambda_S$, called the \textit{interlacement  graph} of $S$ in [1], that has a vertex for each symbol in $S$, and in which two vertices are adjacent if and only if their corresponding symbols are interlaced.  We will denote the adjacency matrix of $\Lambda_S$, regarded as a matrix over $GF(2)$, by $M_S$. This matrix will be called the \textit{interlacement matrix} of $S$.

Each generic smooth immersion $f\colon S^1\to S^2$ generates a double occurrence word, as follows: The image of $f$ contains a finite set of (transverse) double points; assign a unique symbol to each such point. Orient $S^1$, which induces an orientation on the image of $f$. Begin at a point in the image that is not a double point, and trace along the image in the positive direction, as determined by the orientation. Record the double point symbols as you encounter them, stopping when you return to the starting point. The sequence of symbols so produced is a double occurrence word, which we say is \textit{realized} by the immersion $f$.

Gauss [2] addressed the question of which double occurrence words are realizable, and noted that a necessary condition for the realization of a word $S$ is that $\Lambda_S$ be eulerian. Necessary and sufficient conditions for realization have been given by Marx [4], by Lov\'asz and Marx [3], by Rosenstiehl [5], and by de~Fraysseix and Ossona~de~Mendez [1]. 
%See [MR1698548 (2001i:05056b)] for a historical summary and references. 
In this note, we give graph-theoretic and algebraic reformulations of a characterization of realizable double occurrence words that appears in [1]. These reformulations yield Theorem~2 below. The characterization in [1] is itself a restatement of a characterization given in [5]. 

\section{Results}

Let $G$ be a finite simple graph with vertex set $V(G)$, and let $A\subseteq V(G)$ be arbitrary. Following [1], we say that $G$ \textit{satisfies property} $P(A)$ if $G$ is eulerian and, for each pair of distinct vertices $u$ and $v$ in $G$, $u$ and $v$ have an odd number of common neighbors if and only if $u$ and $v$ are neighbors and either both are in $A$ or neither is in $A$. (Note that the requirement that $G$ be eulerian would arise from this condition if we allowed $u=v$.)

Let $\tilde G$ be a finite looped graph without multiple loops or multiple edges. We say that $\tilde G$ is an \textit{orthoprojection graph} if, for each pair of (not necessarily distinct) vertices $u$ and $v$ in $\tilde G$, $u$ and $v$ have an odd number of common neighbors if and only if $u$ and $v$ are neighbors. (Note: A vertex in $\tilde G$ is a neighbor of itself if and only if it is looped.) It is not difficult to check that a graph is an orthoprojection graph if and only if its adjacency matrix, regarded as a matrix over $GF(2)$, is (symmetric and) idempotent. The name ``orthoprojection graph'' reflects the fact that such a matrix represents an orthogonal projection with respect to the standard ``dot product'' over $GF(2)$.

Given a simple graph $G$ and a subset $A\subseteq V(G)$, we can form a looped graph $\tilde G_A$ by placing a loop at each vertex in $A$, in which case we say that $G$ \textit{lifts} to $\tilde G_A$.

Here is our key observation.

\begin{proposition} Let $G$ be a finite simple graph with vertex set $V(G)=\{v_1,\ldots, v_n\}$, and let $A\subseteq V(G)$ be arbitrary. Let $M_G$ be the adjacency matrix of $G$ and let $D_A$ be the diagonal matrix with $d_{ii} =1$ if and only if $v_i\in A$, both regarded as matrices over $GF(2)$. Then the following are equivalent:
\begin{enumerate}
\item $G$ satisfies property $P(A)$.
\item $\tilde G_A$ is an orthoprojection graph.
\item $M_G+D_A$ is idempotent over $GF(2)$.
\end{enumerate}
\end{proposition}

\begin{proof} For an arbitrary vertex $u$, let $n(u)$ denote the neighborhood of $u$ in $G$, and let $\tilde n(u)$ denote the neighborhood of $u$ in $\tilde G_A$. Note that the parity of $|\tilde n(u)|$ differs from that of $|n(u)|$ if and only if $u\in A$. For distinct vertices $u$ and $v$, the parity of $|\tilde n(u)\cap\tilde n(v)|$ differs from that of $|n(u)\cap n(v)|$ if and only if $u$ and $v$ are neighbors and one is a member of $A$ while the other is not. From here it is not difficult to verify that a) and b) are equivalent. That either of these properties is equivalent to c) follows from the fact that $|\tilde n(u)\cap\tilde n(v)|$ is equal (mod 2) to the corresponding entry in $(M_G+D_A)^2$.
\end{proof}

Combining this observation with the restatement of Rosenstiehl's Theorem that appears in the proof of Theorem~10 in [1], we have

\begin{theorem} Let $S$ be a double occurrence word, with interlacement graph $\Lambda_S$ and interlacement matrix $M_S$. Then the following are equivalent:
\begin{enumerate}
\item $S$ is realizable.
\item $\Lambda_S$ lifts to an orthoprojection graph.
\item There exists a diagonal matrix $D_S$ such that $M_S + D_S$ is idempotent over $GF(2)$.
\end{enumerate}

\end{theorem}

\section{Two Remarks} 

\noindent \textbf{Remark.} 
In light of the result above, it is natural to ask if every orthoprojection graph arises by lifting the interlacement graph of a (necessarily realizable) double occurrence word. The answer is ``no.'' The nine-vertex graph depicted in Figure~1 is an orthoprojection graph, but is not the interlacement graph of a double occurrence word. See p.~86 of [6]. This graph is a minimal example; each orthoprojection graph with fewer than nine vertices is the lift of an interlacement graph. See Section~8 of [6] for depictions and further discussion.

\begin{figure}[htb]
\begin{center}
\includegraphics[width=1.1in]{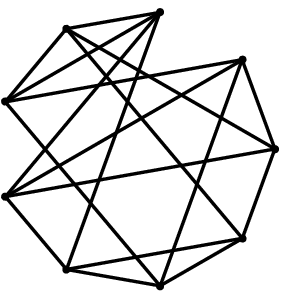}
\caption{}
\label{}
\end{center}
\end{figure}

\noindent \textbf{Remark.} 
It is possible to give a self-contained topological proof of the equivalence of a) and c) in the statement of Theorem~2 above. Indeed, a proof that a) implies c) already appears in [7]. A key observation is that each generic smooth immersion $f\colon S^1\to S^2$ with $n$ double points gives rise to an orthogonal direct sum decomposition of an $n$-dimensional vector space over $GF(2)$, as follows:
From $S^2$, remove a small open disk centered at each of the $n$ double points. This yields a compact surface whose boundary is the disjoint union of $n$ circles. On each of these boundary circles, identify points antipodally. (If necessary, isotope $f$ so that its image meets each circle in two pairs of antipodal points.) The result is then a closed non-orientable surface $\Sigma$ of genus $n$, and the image of the original curve is now embedded and separating in $\Sigma$.
An application of the Mayer-Vietoris homology sequence with coefficients taken in $GF(2)$ then provides the direct sum decomposition described above. In short, the checkerboard splitting of $S^2$ into white and black regions that is produced by the image of $f$ yields an algebraic splitting of an $n$-dimensional vector space over $GF(2)$. Such an algebraic splitting yields a pair of orthogonal projections, and thus a pair of related orthoprojection graphs. One then verifies that the resulting orthoprojection graphs are lifts of the interlacement graph of the original immersion $f$.

A similar topological approach can be taken to give a constructive, inductive proof that c) implies a).

\end{document}